\documentclass{svmult}

\renewcommand{\email}[1]{\emailname: #1} 

\usepackage{mathptmx}       
\usepackage{helvet}         
\usepackage{courier}        

\usepackage{makeidx}         
\usepackage{graphicx}        
\usepackage[bottom]{footmisc}

\usepackage{latexsym}
\usepackage{amsmath}
\usepackage{amsfonts}
\usepackage{amssymb}
\usepackage{bm}

\usepackage{url}
\usepackage{algorithm}
\usepackage{algorithmic}
\usepackage[misc,geometry]{ifsym}

\renewenvironment{proof}{\noindent{\itshape Proof.}}{\smartqed\qed}

\spdefaulttheorem{assumption}{Assumption}{\upshape \bfseries}{\itshape}
\spdefaulttheorem{algo}{Algorithm}{\upshape \bfseries}{\itshape}

\newcommand{\bsn}{{\boldsymbol{n}}}

\newcommand{\bss}{{\boldsymbol{s}}}

\newcommand{\bsx}{{\boldsymbol{x}}}
\newcommand{\bsy}{{\boldsymbol{y}}}
\newcommand{\bsz}{{\boldsymbol{z}}}

\newcommand{\rmd}{{\mathrm{d}}}

\newcommand{\bbH}{{\mathbb{H}}}

\newcommand{\bbM}{{\mathbb{M}}}

\newcommand{\bbP}{{\mathbb{P}}}

\newcommand{\bbS}{{\mathbb{S}}}

\newcommand{\N}{{\mathbb{N}}} 
\newcommand{\R}{{\mathbb{R}}} 

\DeclareSymbolFont{bbold}{U}{bbold}{m}{n}
\DeclareSymbolFontAlphabet{\mathbbold}{bbold}

\newcommand{\calA}{{\mathcal{A}}}

\newcommand{\calG}{{\mathcal{G}}}

\newcommand{\calL}{{\mathcal{L}}}

\newcommand{\calN}{{\mathcal{N}}}

\newcommand{\calS}{{\mathcal{S}}}

\newcommand{\dist}{\ensuremath{\mathrm{dist}}}
\newcommand{\inpro}[2]{\left\langle{#1},{#2}\right\rangle}
\newcommand{\inprod}[2]{\left\langle{#1},{#2}\right\rangle}

\begin{document}
\titlerunning{BVP on spheres}
\title*{Numerical Solutions of a Boundary Value Problem on the Sphere Using Radial Basis Functions}

\author{Q.~T.~Le Gia}

\institute{
 Q.~T.~Le Gia (\Letter)
 \at University of New South Wales, Sydney, Australia \\
 \email{qlegia@unsw.edu.au}
}

\maketitle

\index{Le Gia, Q.~T.}

\abstract{
Boundary value problems on the unit sphere arise naturally in geophysics
and oceanography when scientists model a physical quantity on large scales. 
Robust numerical methods play an important role in solving these problems.
In this article, we construct numerical solutions to a boundary value problem 
defined on a spherical sub-domain (with a sufficiently smooth boundary) 
using radial basis functions (RBFs). The error analysis 
between the exact solution and the approximation is provided.
Numerical experiments are presented to confirm theoretical estimates.
}

\section{Introduction}\label{sec:1}
Boundary value problems on the unit sphere arise naturally in geophysics 
and oceanography when scientists model a physical quantity on large scales. 
In that situation, the curvature of the Earth cannot be ignored, and a boundary 
value problem has to be formulated on a subdomain of the unit sphere. For example, 
the study of planetary-scale oceanographic flows in which oceanic eddies interact 
with topography such as ridges and land masses or evolve in closed basin lead to 
the study of point vortices on the surface of the sphere with walls \cite{gill,chaos}.
Such vortex motions can be described as a Dirichlet problem on a subdomain
of the sphere for the Laplace-Beltrami operator \cite{crowdy,kidambi-newton}.
Solving the problem exactly via conformal mapping methods onto the complex plane 
was proposed by Crowdy in \cite{crowdy}. Kidambi and Newton \cite{kidambi-newton} 
also considered such a problem, assuming the sub-surface of the sphere lent itself
to method of images. A boundary integral method for constructing numerical 
solutions to the problem was discussed in \cite{gemmrich}. 
n this work, we propose a collocation method using spherical radial basis
functions. Radial basis functions (RBFs) present a simple and effective way to 
construct approximate solutions to partial differential equations (PDEs) on spheres, via 
a collocation method \cite{morton-neamtu} or a Galerkin method \cite{LeG04}. They have been 
used successfully for solving transport-like equations on the sphere
\cite{Flyer-Wright-07-1,Flyer-Wright-09-1}. The method does not require a mesh, and is simple 
to implement. 

While meshless methods using RBFs have been employed to derive numerical
solutions for PDEs on the sphere only recently, it should be mentioned that approximation 
methods using RBFs for PDEs on bounded domains have been around for the last two decades. 
Originally proposed by Kansa \cite{Kan90,Kan90a} for fluid dynamics, approximation
methods for many types of PDEs defined on bounded domains in $\R^n$ using
RBFs have since been used widely \cite{Fas99,franke-schback-1998a,HonMao98,HonSch01}. 

To the best of our knowledge, approximation methods using RBFs have not been 
investigated for boundary value problems defined on subdomains of the unit sphere.
Given the potential of RBF methods on these problems, the present paper aims
to present a collocation method for boundary value problems on the sphere
and provide a mathematical foundation for error estimates. 

The paper is organized as follows: in Section~\ref{prelim} we review some preliminaries
on functions spaces, positive definite kernels, radial basis functions and the generalized
interpolation problem on discrete point sets on the unit sphere. In Section~\ref{bvp-section} 
we define the boundary value problem on a spherical cap, then present a collocation method using
spherical radial basis functions and our main result, Theorem~\ref{bvp-theorem}. 
We conclude the paper 
by giving some numerical experiments in the last section.

Throughout the paper, we denote by $c,c_1,c_2,\ldots$ generic positive constants that 
may assume different values at different places, even within the same formula.

For two sequences $\{a_\ell\}_{\ell\in\N_0}$ and $\{b_\ell\}_{\ell\in\N_0}$,
the notation $a_\ell\sim b_\ell$ means that there exist positive constants 
$c_1$ and $c_2$ such that $c_1 b_\ell \leq a_\ell \leq c_2 b_\ell$ 
for all $\ell\in\N_0$.


\section{Preliminaries}\label{prelim}
Let $\bbS^n$ be the {\em unit sphere}, i.e.
$ 
    \bbS^n := \left\{ \bsx\in\R^{n+1}\,:\, \|\bsx\| = 1 \right\}
$
in the Euclidean space $\R^{n+1}$, where $\|\bsx\|:= \sqrt{\bsx\cdot\bsx}$
denotes the Euclidean norm of $\R^{n+1}$, induced by the Euclidean inner
product $\bsx\cdot\bsy$ of two vectors $\bsx$ and $\bsy$ in $\R^{n+1}$.
The surface area of the unit sphere $\bbS^n$ is denoted by $\omega_n$ and is 
given by 
\[
    \omega_n := |\bbS^n| = \frac{2 \pi^{(n+1)/2}}{\Gamma((n+1)/2)}.
\]

The {\em spherical distance}\/ (or geodesic distance) $\dist_{\bbS^n}(\bsx,\bsy)$ of two points 
$\bsx\in\bbS^n$ and $\bsy\in\bbS^n$ is defined as the length of a shortest geodesic 
arc connecting the two points. The geodesic distance $\dist_{\bbS^n}(\bsx,\bsy)$ is the 
angle in $[0,\pi]$ between the points $\bsx$ and $\bsy$, thus 
\[
    \dist_{\bbS^n}(\bsx,\bsy) := \arccos(\bsx\cdot\bsy).
\]

Let $\Omega$ be an open simply connected subdomain of the sphere.
For a point set $X:=\{\bsx_1,\bsx_2,\ldots,\bsx_N\}\subset\bbS^n$,  
the {\em (global) mesh norm}\/ $h_X$ is given by 
\[
    h_X = h_{X,\bbS^n}:= \sup_{\bsx\in\bbS^n} \inf_{\bsx_j\in X} \dist_{\bbS^n}(\bsx,\bsx_j),
\]  
and the {\em local mesh norm}\/ $h_{X,\Omega}$ with respect to
the subdomain $\Omega$ is defined by 
\[
    h_{X,\Omega} := \sup_{\bsx\in \Omega} 
    \inf_{\bsx_j\in X\cap \Omega} \dist_{\bbS^n}(\bsx,\bsx_j).
\]
The mesh norm $h_{X_2,\partial \Omega}$ of $X_2\subset\partial \Omega$ 
along the boundary $\partial \Omega$ is defined by 
\begin{equation}
\label{boundary-mesh-norm}
    h_{X_2,\partial \Omega} := \sup_{\bsx\in\partial \Omega}
    \inf_{\bsx_j\in X_2} \dist_{\partial \Omega}(\bsx,\bsx_j),
\end{equation}
where $\dist_{\bsx\in\partial \Omega}$ is here 
the geodesic distance along the boundary $\partial \Omega$. 

%
\subsection{Sobolev spaces on the sphere}\label{fct-subsection}
Let $\Omega$ be $\bbS^n$ or an open measurable subset of $\bbS^n$.
Let $L_2(\Omega)$ denote the Hilbert space of (real-valued) 
square-integrable functions on $\Omega$ with the inner product
\[
    \inprod{f}{g}_{L_2(\Omega)} := 
    \int_{\Omega} f(\bsx) g(\bsx) \rmd\omega_n(\bsx)
\]
and the induced norm $\|f\|_{L_2(\Omega)}:= \inprod{f}{f}_{L_2(\Omega)}^{1/2}$.
Here $\rmd\omega_n$ is the Lebesgue surface area element
of the sphere $\bbS^n$.

The space of continuous functions on the sphere $\bbS^n$ and on 
the closed subdomain $\overline{\Omega}$ are denoted 
by $C(\Omega)$ and $C(\overline{\Omega})$ and are endowed with the 
supremum norms 
\[
    \|f\|_{C(\bbS^n)} := \sup_{\bsx\in\bbS^n}|f(\bsx)| 
    \qquad\mbox{and}\qquad
    \|f\|_{C(\overline{\Omega})} := \sup_{\bsx\in \overline{\Omega}} |f(\bsx)|,
\]
respectively.

A {\em spherical harmonic}\/ of degree $\ell\in\N_0$ (for the sphere $\bbS^n$) is 
the restriction of a homogeneous harmonic polynomial on $\R^{n+1}$ 
of exact degree $\ell$ to the unit sphere $\bbS^n$. The vector space of all 
spherical harmonics of degree $\ell$ (and the zero function) is denoted 
by $\bbH_\ell(\bbS^n)$ and has the dimension $Z(n,\ell):=\dim(\bbH_\ell(\bbS^n))$
given by 
\[
    Z(n,0)=1 \qquad\mbox{and}\qquad Z(n,\ell)=
    \frac{(2\ell+n-1)\Gamma(\ell+n-1)}{\Gamma(\ell+1)\Gamma(n)} 
    \quad\mbox{ for } \ell\in\N. 
\]  
By $\{ Y_{\ell,k} \,:\, k=1,2,\ldots,Z(n,\ell)\}$, 
we will always denote an 
$L_2(\bbS^n)$-orthonormal basis of $\bbH_\ell(\bbS^n)$
consisting of spherical harmonics of degree $\ell$.  
Any two spherical harmonics of different degree are orthogonal to each other, 
and the union of all sets $\{ Y_{\ell,k} \,:\, k=1,2,\ldots,Z(n,\ell)\}$ 
constitutes a complete orthonormal system for $L_2(\bbS^n)$. 
Thus any function $f \in L_2(\bbS^n)$ can be represented
in $L_2(\bbS^n)$-sense by its {\em Fourier series}\/ (or Laplace series) 
\[
    f = \sum_{\ell=0}^\infty \sum_{k=1}^{Z(n,\ell)} 
    \widehat{f}_{\ell,k} Y_{\ell,k},
\]
with the Fourier coefficients $\widehat{f}_{\ell,k}$ defined by 
\[
    \widehat{f}_{\ell,k} := \int_{\bbS^n} f(\bsx) Y_{\ell,k}(\bsx) \rmd\omega_n(\bsx).
\] 

The {\em space of spherical polynomials of degree $\leq K$}\/ (that is, 
the set of the restrictions to $\bbS^n$ of all polynomials on $\R^{n+1}$
of degree $\leq K$) is denoted by $\bbP_K(\bbS^n)$. We have 
$\bbP_K(\bbS^n)=\bigoplus_{\ell=0}^K \bbH_\ell(\bbS^n)$ and 
$\dim(\bbP_K(\bbS^n))=Z(n+1,K)\sim(K+1)^n$.  

Any orthonormal basis $\{ Y_{\ell,k} \,:\, k=1,2,\ldots,Z(n,\ell)\}$ 
of $\bbH_\ell(\bbS^n)$ satisfies the {\em addition theorem}\/ (see~\cite[p.10]{muller}) 
\begin{equation}
\label{addition}
    \sum_{k=0}^{Z(n,\ell)} Y_{\ell,k}(\bsx) Y_{\ell,k}(\bsy) = 
    \frac{Z(n,\ell)}{\omega_n} P_{\ell}(n+1;\bsx\cdot\bsy),  
\end{equation}
where $P_{\ell}(n+1;\cdot)$ is the {\em normalized Legendre polynomial}\/ 
of degree $\ell$ in $\R^{n+1}$. The normalized Legendre polynomials 
$\{P_{\ell}(n+1;\cdot)\}_{\ell\in\N_0}$,
form a complete orthogonal system for the space $L_2([-1,1];(1-t^2)^{(n-2)/2})$ of 
functions on $[-1,1]$ which are square-integrable with respect to the 
weight function $w(t):=(1-t^2)^{(n-2)/2}$. They satisfy $P_{\ell}(n+1;1)=1$ and
\begin{equation}
\label{orthogonal}
    \int_{-1}^{+1} P_\ell(n+1;t)P_{k}(n+1;t)(1-t^2)^{(n-2)/2} \rmd t = 
    \frac{\omega_n}{\omega_{n-1}Z(n,\ell)} \delta_{\ell,k},
\end{equation}
where $\delta_{\ell,k}$ is the Kronecker delta (defined to be one if
$\ell=k$ and zero otherwise). 

The {\em Laplace-Beltrami operator}\/  
$\Delta^\ast$ (for the unit sphere $\bbS^n$) is the angular 
part of the Laplace operator $\Delta=\sum_{j=1}^{n+1} \partial^2/\partial x_j^2$ 
for $\R^{n+1}$. Spherical harmonics of degree $\ell$ on $\bbS^n$
are eigenfunctions of $-\Delta^\ast$, more precisely,  
\[
    - \Delta^\ast Y_{\ell} = \lambda_\ell Y_{\ell}
    \quad\mbox{for all}\ Y_\ell\in\bbH_\ell(\bbS^n)
    \qquad\mbox{with}\qquad  
    \lambda_\ell := \ell(\ell+n-1).
\]

For $s\in\R_0^+$, the {\em Sobolev space}\/ $H^s(\bbS^n)$ 
is defined by (see \cite[Chapter~1, Remark~7.6]{LioMag})
\[
    H^s(\bbS^n) := \left\{ f \in L_2(\bbS^n) : 
    \sum_{\ell=0}^\infty (1+\lambda_\ell)^s \sum_{k=1}^{Z(n,\ell)} 
    |\widehat{f}_{\ell,k}|^2  < \infty \right\}.  
\]
The space $H^s(\bbS^n)$ is a Hilbert space with the inner product 
\[
    \inprod{f}{g}_{H^s(\bbS^n)} := \sum_{\ell=0}^\infty 
    (1+\lambda_\ell)^s \sum_{k=1}^{Z(n,\ell)} 
    \widehat{f}_{\ell,k} \widehat{g}_{\ell,k}
\] 
and the induced norm 
\begin{equation}
\label{Hs-norm}
    \|f\|_{H^s(\bbS^n)} := \inprod{f}{f}_{H^s(\bbS^n)}^{1/2}
    = \sum_{\ell=0}^\infty (1+\lambda_\ell)^s  
    \sum_{k=1}^{Z(n,\ell)} |\widehat{f}_{\ell,k}|^2. 
\end{equation} 

If $s>n/2$, then $H^s(\bbS^n)$ is embedded into $C(\bbS^n)$,
and the Sobolev space $H^s(\bbS^n)$ is a 
{\em reproducing kernel Hilbert space}\/. This means that there 
exists a kernel $K_s:\bbS^n\times\bbS^n\rightarrow\R$, the so-called 
reproducing kernel, with the following properties: 
(i) $K_s(\bsx,\bsy) = K_s(\bsy,\bsx)$ for all $\bsx,\bsy\in\nolinebreak[4]\bbS^n$,
(ii)~$K_s(\cdot,\bsy)\in H^s(\bbS^n)$ for all (fixed) $\bsy\in\bbS^n$,
and (iii) the reproducing property 
\[
    \inprod{f}{K_s(\cdot,\bsy)}_{H^s(\bbS^n)} = f(\bsy)
    \qquad\mbox{for all}\ f\in H^s(\bbS^n)
    \ \mbox{and all}\ \bsy\in\bbS^n.
\]

Sobolev spaces on $\bbS^n$ can also be defined using local charts 
(see \cite{LioMag}). Here we use a specific atlas
of charts, as in \cite{HubMor04}. 

Let $\bsz$ be a given point on $\bbS^n$, the spherical cap centered
at $\bsz$ of radius $\theta$ is defined by
\begin{eqnarray*}
G(\bsz,\theta) &=& \{ \bsy \in \bbS^{n} : \cos^{-1}(\bsz \cdot \bsy) < \theta \},  \qquad \theta \in (0, \pi),  
\end{eqnarray*}
where $\bsz\cdot \bsy$ denotes the Euclidean inner product of $\bsz$~and $\bsy$
in $\R^{n+1}$.

Let $\hat{\bsn}$ and $\hat{\bss}$ denote the north 
and south poles of $\bbS^n$, respectively. Then a simple cover 
for the sphere is provided by
\begin{equation}\label{def_Ui}
  U_1 = G(\hat{\bsn},\theta_0) \quad \mbox{and}\quad 
  U_2 = G(\hat{\bss},\theta_0),
  \mbox{ where } \theta_0 \in (\pi/2,2\pi/3).                                                                     
\end{equation}
The stereographic projection $\sigma_{\hat{\bsn}}$ of the punctured sphere                                           
$\bbS^n\setminus\{\hat{\bsn}\}$ onto $\R^n$ is defined as a mapping that maps                                           
$\bsx \in \bbS^n \setminus \{\hat{\bsn}\}$ to the intersection of the equatorial                                           
hyperplane $\{\bsz=0\}$ and the extended line that passes through $\bsx$ and $\hat{\bsn}$. The stereographic projection $\sigma_{\hat{\bss}}$ based on $\hat{\bss}$ can be defined analogously. We set                                                                                       
\begin{equation}\label{def_phi}                                                                                                   
  \psi_1 = \frac{1}{\tan(\theta_0/2)}\sigma_{\hat{\bss}}|_{U_1} 
  \quad \mbox{ and }\quad                              
  \psi_2 = \frac{1}{\tan(\theta_0/2)}\sigma_{\hat{\bsn}}|_{U_2},                                                     
\end{equation}
so that $\psi_k$, $k=1,2$, maps $U_k$ onto $B(0,1)$, the unit ball in $\R^n$.                                     
We conclude that $\calA =\{U_k,\psi_k\}_{k=1}^2$ is a $C^\infty$ atlas of                                         
covering coordinate charts for the sphere. It is known (see \cite{Rat94})                                     
that the stereographic coordinate charts $\{\psi_k\}_{k=1}^2$ as defined in                                       
(\ref{def_phi}) map spherical caps to Euclidean balls, but in general                                             
concentric spherical caps are not mapped to concentric Euclidean balls.                                           
The projection $\psi_k$, for $k=1,2$, does not distort too much the geodesic                                      
distance between two points $\bsx,\bsy \in \bbS^n$, as shown in \cite{LeGNarWarWen06}.

With the atlas so defined, we define the map $\pi_k$ which takes a real-valued                                    
function $g$ with compact support in $U_k$ into a real-valued function on $\R^n$ by                               
\[
 \pi_k(g) (x) = \left\{  \begin{array}{ll}                                                                        
                            g \circ \psi^{-1}_k(x),& \mbox{ if } x\in B(0,1),\\
                            0, & \mbox{ otherwise }.
                         \end{array}             
                \right. 
\]
Let $\{\chi_k:\bbS^n \rightarrow \R \}_{k=1}^2$ be a partition of unity subordinated to the atlas, i.e., a pair of non-negative infinitely differentiable functions $\chi_k$ on $\bbS^n$ with compact support in $U_k$, 
such that $\sum_{k}\chi_k = 1$. For any function $f:\bbS^n \rightarrow \R$, we can use the partition of unity to write
\[
  f = \sum_{k=1}^2 (\chi_k f), 
  \mbox{ where } (\chi_k f)(\bsx) = \chi_k(\bsx)f(\bsx), \quad \bsx \in \bbS^n.   
\]
The Sobolev space $H^s(\bbS^n)$ is defined to be the set
\[
  \left\{f \in L_2(\bbS^n) : 
  \pi_k(\chi_k f) \in H^s(\R^n) \quad\mbox{ for } k=1,2 \right\},
\]
which is equipped with the norm
\begin{equation}\label{defH1atlas}
  \|f\|_{H^s(\bbS^n)} = 
  \left( \sum_{k=1}^2 \|\pi_k (\chi_k f)\|^2_{H^s(\R^n)} \right)^{1/2}.
\end{equation}
This $H^s(\bbS^n)$ norm is equivalent to the $H^s(\bbS^n)$ norm 
given previously in \eqref{Hs-norm} (see \cite{LioMag}).

Let $\Omega \subset \bbS^n$ be an open connected set with sufficiently smooth boundary. In order to define the Sobolev spaces on $\Omega$, let
$
  D_k = \psi_k( \Omega \cap U_k)\text{ for } k=1,2.
$
The local Sobolev space $H^\tau(\Omega)$ is defined to be the set
\[
 f \in L_2(\Omega) : \pi_k(\chi_k f)|_{D_k} \in H^s(D_k)
 \text{ for } k=1,2, \; D_k \neq \emptyset,
\]
which is equipped with the norm
\begin{equation}\label{def:local Sob}
\|f\|_{H^s(\Omega)} = 
\left( \sum_{k=1}^2 \|\pi_k(\chi_k f)|_{D_k}\|^2_{H^s(D_k)}\right)^{1/2}
\end{equation}
where, if $\Omega=\emptyset$, then we adopt the convention that
$\|\cdot\|_{H^s(D_k)} = 0$.

It should be noted that if $s=m$ which is a positive integer,
we can define the local Sobolev norm via the following formula
\begin{equation}\label{def:local Sob 2}
 \|f\|_{H^m(\Omega)} =  \left( \sum_{k=0}^m 
 \inprod{\nabla^k f}{\nabla^k f}_{L_2(\Omega)} \right)^{1/2},
\end{equation}
where $\nabla$ is the surface gradient on the sphere.


Now we state an extension theorem for a local domain on the sphere.
We follow a framework set out in \cite[Chapter 4.4]{michael-taylor-1}. 
To this end, let us consider Sobolev spaces $H^s(\R^n_{+})$, with
$\R^n_{+} = \{ \bsx \in \R^n: x_1 > 0 \}$. For $k \ge 0$ an integer,
let
\[
  H^k(\R^n_{+}) = \{u \in L^2(\R^n_{+}) : D^\alpha u \in L^2(\R^n_{+})
         \text{ for } |\alpha | \le k\}.
\]
Here, $D^\alpha u$ is considered as a distribution on the interior $\R^n_{+}$.
We claim that each $u \in H^k(\R^n_{+})$ is the restriction to $\R^n_{+}$ of
an element of $H^k(\R^n)$. To see this, fix an integer $N$, for an
$u \in \calS(\overline{\R^n_{+}})$ let
\[
  Eu(x) = \begin{cases}  
          u(x)  & \text{ for } x_1 \ge 0, \\
         \sum_{j=1}^N a_j u(-j x_1, \bsx'), & \text{ for } x_1 <0. 
      \end{cases}
\]
\begin{lemma}
One can pick the coefficients $a_1,\ldots,a_N$ such that the map $E$
as a unique continuous extension to
\[
  E : H^k(\R^n_{+}) \rightarrow H^k(\R^n), \text{ for } k \le N-1.
\]
\end{lemma}
\begin{proof}
Given $u \in \calS(\R^n)$, we get an $H^k$-estimate on $Eu$ provided all the derivatives
of $Eu$ of order $N-1$ match up at $x_1 = 0$, that is, provided 
\begin{equation}\label{linsys}
\sum_{j=1}^N (-j)^\ell a_j = 1, \text{ for } \ell=0,1,\ldots,N-1.
\end{equation}
The system \eqref{linsys} is a a linear system of $N$ equations for $N$ unknowns $a_j$;
its determinant is a Vandermonde determinant that is non-zero, so $a_j$ can be found.
\end{proof}

Now for $k \ge 0$ being an integer, let $H^k(\Omega)$ be the space of all 
$u \in L^2(\Omega)$ such that $Pu \in L^2(\Omega)$ for all differential operators $P$
of order $\le k$ with coefficients in $C^\infty(\overline{\Omega})$. By covering
a neighbourhood of $\partial \Omega \subset \bbS^n$ with coordinate patches and locally
using the extension operator $E$ from above, we get, for each finite $N$, an extension
operator
\begin{equation}\label{ext-op-k}
    E: H^k(\Omega) \rightarrow H^k(\bbS^n), \quad 0 \le k \le N-1.
\end{equation}
For real $s\ge 0$, we can use interpolation between Banach spaces 
(see \cite[Chapter 4.2]{michael-taylor-1}) to define
\begin{equation}\label{ext-op-s}
    E: H^s(\Omega) \rightarrow H^s(\bbS^n).
\end{equation}

\begin{theorem}[Trace theorem]\label{trace}
Let $\Omega \subset \bbS^n$ be a local region with a sufficient smooth 
boundary. Then, for $s>1/2$, the restriction of $f\in H^{s}(\Omega)$ to $\partial\Omega$
is well defined, belongs to $H^{s-1/2}(\partial\Omega)$, and satisfies
\[
\|f\|_{H^{s-1/2}(\partial\Omega)} \le C \|f\|_{H^s(\Omega)}.
\]
\end{theorem}
\begin{proof}
The boundary $\partial D_k$ of $D_k = \psi_k(\Omega \cap U_k)$ is
given by $\psi_k(\partial \Omega \cap U_k$ for $k=1,2$.
Then,
\[
 \|f\|^2_{H^{s-1/2}(\partial\Omega)} = 
 \sum_{k=1}^2 
 \|(\pi_k \chi_k f)|_{\partial D_k} \|^2_{H^{s-1/2}(\partial D_k)}.
\]
Using the trace theorem for bounded domains in $\R^n$ \cite[Theorem 8.7]{wloka}, there are constants $c_k>0$ for $k=1,2$ so that
\[
 \|(\pi_k \chi_k f)|_{\partial D_k} \|_{H^{s-1/2}(\partial D_k)}
 \le c_k \|(\pi_k \chi_k f)|_{D_k} \|_{H^{s}(D_k)}.
\]
Hence
\[
\|f\|^2_{H^{s-1/2}(\partial\Omega)}
\le \max\{c_1^2,c_2^2\} 
\sum_{k=1}^2 \|(\pi_k \chi_k f)|_{D_k} \|^2_{H^{s}(D_k)}
= \max\{c_1^2,c_2^2\}\|f\|^2_{H^s(\Omega)}.
\]
\end{proof}
%
\subsection{Positive definite kernels on the sphere and native spaces}
\label{sphbasfct-subsection}
A continuous real-valued kernel $\phi:\bbS^n\times\bbS^n\rightarrow\R$ is called 
\emph{positive definite}\/ on $\bbS^n$ if 
(i) $\phi(\bsx,\bsy)=\phi(\bsy,\bsx)$ for all $\bsx,\bsy\in\bbS^n$ and (ii) 
for every finite set of distinct points $X = \{\bsx_1,\bsx_2,\ldots,\bsx_N\}$ 
on $\bbS^n$, the symmetric matrix $[\phi(\bsx_i,\bsx_j)]_{i,j=1,2,\ldots,N}$
is positive definite. 

A kernel $\phi:\bbS^n\times\bbS^n\rightarrow\R$ defined via 
$\phi(\bsx,\bsy) := \Phi(\bsx\cdot \bsy)$, $\bsx,\bsy\in\bbS^n$,
with a univariate function $\Phi$, is called a {\em zonal}\/ kernel. 

Since the normalized Legendre polynomials $\{P_{\ell}(n+1;\cdot)\}_{\ell\in\N_0}$,
form a complete orthogonal system for $L_2([-1,1];(1-t^2)^{(n-2)/2})$,
any function $\Phi\in L_2([-1,1];(1-t^2)^{(n-2)/2})$ 
can be expanded into a {\em Legendre series}\/ (see (\ref{orthogonal})
for the normalization)
\begin{equation}
\label{Phi-expansion}
    \Phi(t) = \frac{1}{\omega_{n}} \sum_{\ell=0}^\infty 
    a_\ell Z(n,\ell) P_{\ell}(n+1;t), 
\end{equation}
with the Legendre coefficients
\[
    a_\ell := \omega_{n-1} \int_{-1}^{+1} \Phi(t) 
    P_\ell(n+1;t) (1-t^2)^{(n-2)/2} \rmd t.
\]
Due to (\ref{Phi-expansion}) and the addition theorem (\ref{addition}), 
a zonal kernel 
$\phi(\bsx,\bsy) := \Phi(\bsx\cdot \bsy)$, $\bsx,\bsy\in\bbS^n$,
where $\Phi\in L_2([-1,1];(1-t^2)^{(n-2)/2})$,
has the expansion
\begin{equation}
\label{phi-expansion}
    \phi(\bsx,\bsy) = \frac{1}{\omega_n} \sum_{\ell=0}^\infty
    a_\ell Z(n,\ell) P_\ell(n+1;\bsx\cdot\bsy)
    = \sum_{\ell=0}^\infty \sum_{k=1}^{Z(n,\ell)} 
    a_\ell Y_{\ell,k}(\bsx) Y_{\ell,k}(\bsy).
\end{equation}

In this paper we will only consider 
{\em positive definite zonal continuous kernels}\/ $\phi$ of the 
form (\ref{phi-expansion}) for which 
\begin{equation}
\label{kernel-condition}
\sum_{\ell=0}^\infty |a_\ell| Z(n,\ell) < \infty.
\end{equation}
This condition implies that the sums in (\ref{phi-expansion}) converge uniformly.

In \cite{chen-menegatto-sun}, a complete characterization of positive 
definite kernels is established: a kernel $\phi$ of the form 
(\ref{phi-expansion}) satisfying the condition (\ref{kernel-condition}) 
is positive definite {\em if and only if}\/ 
$a_\ell \ge 0$ for all $\ell\in\N_0$ 
and $a_\ell > 0$ for infinitely many even values of $\ell$ and
infinitely many odd values of $\ell$ (see also \cite{schoenberg} and
\cite{xu-cheney}).

With each positive definite zonal continuous kernel 
$\phi$ of the form (\ref{phi-expansion}) and satisfying the 
condition (\ref{kernel-condition}), we associate a  
{\em native space}: Consider the linear space 
\[
    F_\phi := \left\{ \sum_{j=1}^N \alpha_j \phi(\cdot,\bsx_j) \,:\,
    \alpha_j\in\R,\ \bsx_j\in\bbS^n,\ j=1,2,\ldots,N; \ N\in\N \right\},
\]
endowed with the inner product 
\[
    \inprod{ \sum_{j=1}^N \alpha_j \phi(\cdot,\bsx_j)} 
    {\sum_{i=1}^M \beta_i \phi(\cdot,\bsy_i)}_\phi :=
    \sum_{j=1}^N \sum_{i=1}^M \alpha_j \beta_i \phi(\bsx_j,\bsy_i)
\]
and the associated norm $\|f\|_\phi:=\inprod{f}{f}_\phi^{1/2}$.
The {\em native space}\/ $\calN_\phi$ associated with $\phi$ 
is now defined as the completion of $F_\phi$ with respect 
to the norm $\|\cdot\|_\phi$. By construction, 
the native space $\calN_\phi$ is a Hilbert space, 
and we will denote its inner product and norm also by 
$\inprod{\cdot}{\cdot}_\phi$ and $\|\cdot\|_\phi$, respectively.

The native space $\calN_\phi$ is a (real) 
{\em reproducing kernel Hilbert space}\/ with the reproducing kernel $\phi$. 
This means that (i) $\phi$ is symmetric,
(ii) $\phi(\cdot,\bsy)\in\calN_\phi$ for all (fixed) $\bsy\in\bbS^n$, and (iii) 
the {\em reproducing property}\/ holds, that is, 
\begin{equation}
\label{reproducing_property}
    \inprod{f}{\phi(\cdot,\bsy)}_\phi = f(\bsy), 
    \qquad\mbox{for all}\ f\in\calN_\phi\ \mbox{and all}\ \bsy\in\bbS^n.
\end{equation} 

It is known that the native space $\calN_\phi$ associated with a positive 
definite continuous zonal kernel $\phi$, given by (\ref{phi-expansion})
and satisfying the conditions (\ref{kernel-condition}) 
and $a_\ell>0$ for all $\ell\in\N_0$, can be described by 
\[
    \calN_\phi = \left\{f \in L_2(\bbS^n) \,:\,
    \sum_{\ell=0}^\infty \sum_{k=1}^{Z(n,\ell)}
    \frac{|\widehat{f}_{\ell,k}|^2}{a_\ell}  < \infty \right\},
\]
equipped with the inner product
\[
    \inprod{f}{g}_\phi = \sum_{\ell=0}^\infty \sum_{k=1}^{Z(n,\ell)}
    \frac{\widehat{f}_{\ell,k} \widehat{g}_{\ell,k}}{a_\ell}
\]
and the associated norm 
\begin{equation}
\label{norm-native-space}
    \|f\|_\phi = \inprod{f}{f}_\phi^{1/2} = 
    \left( \sum_{\ell=0}^\infty \sum_{k=1}^{Z(n,\ell)}
    \frac{|\widehat{f}_{\ell,k}|^2}{a_\ell} \right)^{1/2}.
\end{equation}
If $a_\ell>0$ for all $\ell\in\N_0$,
we can conclude, from the assumption (\ref{kernel-condition}), 
that the Fourier series of any $f\in\calN_\phi$ converges
uniformly and that the native space $\calN_\phi$ is embedded 
into $C(\bbS^n)$.

Comparing (\ref{norm-native-space}) with (\ref{Hs-norm}),
we see that if $a_\ell\sim(1+\lambda_\ell)^{-s}$, then 
$\|\cdot\|_\phi$ and $\|\cdot\|_{H^s(\bbS^n)}$ are equivalent 
norms, and hence $\calN_\phi$ and $H^s(\bbS^n)$ are the same space.  
%
%
\subsection{Generalized interpolation with RBFs}
\label{rbf-interpolation-subsection}
Let $\phi:\bbS^n\times\bbS^n\rightarrow\R$ be a positive definite
zonal continuous kernel given by (\ref{phi-expansion}) and satisfying the 
condition (\ref{kernel-condition}). Since the native space $\calN_\phi$ 
is a reproducing kernel Hilbert space with reproducing kernel $\phi$,
any continuous linear functional $\calL$ on $\calN_\phi$ has the representer 
$\calL_2\phi(\cdot,\cdot)$. (Here the index $2$ in $\calL_2 \phi(\cdot,\cdot)$ 
indicates that $\calL$ is applied to the kernel $\phi$ as a function of 
its second argument. Likewise $\calL_1 \phi(\cdot,\cdot)$ will indicate 
that $\calL$ is applied to the kernel $\phi$ as a function of its first argument.) 

For a linearly independent set $\Xi=\{\calL^1,\calL^2,\ldots,\calL^N\}$
of continuous linear functionals on~$\calN_\phi$, 
the {\em generalized radial basis function (RBF) interpolation problem}\/ 
can be formulated as follows:      
Given the values $\calL^1 f,\calL^2 f,\ldots,\calL^N f$ of a 
function $f\in\calN_\phi$, find the function $\Lambda_{\Xi} f$ in 
the $N$-dimensional approximation space
\[
    V_\Xi := {\ensuremath{\mathrm{span}}}
           \left\{ \calL^j_2\phi(\cdot,\cdot) \,:\, j=1,2,\ldots,N \right\}
\] 
such that the conditions
\begin{equation}
\label{interpol-1}
    \calL^i (\Lambda_\Xi f) = \calL^i f, \qquad i=1,2,\ldots,N,
\end{equation}
are satisfied. We will call the function $\Lambda_\Xi f\in V_\Xi$ the 
{\em radial basis function approximant (RBF approximant)}\/ of $f$.

Writing the RBF approximant $\Lambda_\Xi f$ as 
\[
    \Lambda_{\Xi} f(\bsx) = \sum_{j=1}^N \alpha_j \calL^j_2\phi(\bsx,\cdot),
    \qquad \bsx\in\bbS^n,
\]
the interpolation conditions (\ref{interpol-1}) can therefore be written as
\begin{equation}
\label{linear-system}
    \sum_{j=1}^N \alpha_j 
    \inprod{\calL^j_2 \phi(\cdot,\cdot)}{\calL^i_2\phi(\cdot,\cdot)}_\phi 
    = \sum_{j=1}^N \alpha_j \calL^i_1 \calL^j_2 \phi(\cdot,\cdot) \\
    = \calL^i f, \qquad i=1,2,\ldots,N. 
\end{equation}
Since $f\in\calN_\phi$, we have $\calL^i f = \inprod{f}{\calL^i_2\phi(\cdot,\cdot)}_\phi$, 
$i=1,2,\ldots,N$, and  we see that $\Lambda_\Xi f$ is just the 
{\em orthogonal projection}\/ of $f\in\calN_\phi$
onto the approximation space $V_\Xi$ with respect to $\inprod{\cdot}{\cdot}_\phi$. 
Therefore,
\begin{equation}\label{pythagoras}
\| f - \Lambda_{\Xi} f\|_\phi \le \|f\|_\phi.
\end{equation}
The linear system has always a unique solution, because its matrix 
$$[\calL^i_1 \calL^j_2 \phi(\cdot,\cdot)]_{i,j=1,2,\ldots,N}$$ is the 
Gram matrix of the representers of the linearly independent 
functionals in $\Xi$. 

We observe here that the linear system (\ref{linear-system})
can be solved for any given data set $\{\calL^i f\,:\,i=1,2,\ldots,N\}$,
where the data does not necessarily has to come from a function in the 
native space $\calN_\phi$, but may come from any function $f$ 
for which $\calL^i f$ is well-defined for all $i=1,2,\ldots,N$.
Even if $f$ is not in the native space we will use the notation $\Lambda_\Xi f$
for the solution of the generalized RBF interpolation problem (\ref{interpol-1}).
%
%
\subsection{Sobolev bounds for functions with scattered zeros}
We need the following results from \cite{HanNarWar12} concerning
functions with scattered zeros on a subdomain of a Riemannian manifold.
\begin{theorem}\label{thm:HNW12}
Let $\bbM$ be a Riemannian manifold, $\Omega \subset \bbM$ be a bounded,
Lipschitz domain that satisfies a certain uniform cone condition. Let
$X$ be a discrete set with sufficiently small mesh norm $h$. If
$u\in W^m_p(\Omega)$ satisfies $u|_X = 0$, then we have
\[
  \|u\|_{W^k_p(\Omega)} \le C_{m,k,p,\bbM} h^{m-k} \|u\|_{W^m_p(\Omega)}
\]
and
\[
 \|u\|_{L^\infty(\Omega)} \le C_{m,k,p,\bbM} h^{m-d/p} \|u\|_{W^m_p(\Omega)}.
\]
\end{theorem}
\section{Boundary value problems on the sphere}\label{bvp-section}

After all these preparations we can formulate a boundary value
problem for an elliptic differential operators $L$. Our standard
application (and numerical example in Section~\ref{numerical-section})
will be $L=\kappa^2 I - \Delta^\ast$, where $I$ is the identity operator
and $\kappa$ is some fixed constant, on simply connected subregion $\Omega$ 
on $\bbS^n$ with a Lipschitz boundary $\partial \Omega$.
This partial differential equation occurs, for example, when 
solving the heat equation and the wave equation with separation 
of variables (for $\kappa \ne 0$) or in studying the vortex motion on the sphere
(for $\kappa = 0$). 

Let $s>2$, and let $\Omega$ be a simply connected subregion with a Lipschitz
boundary. Assume that 
the functions $f \in W_2^{s-2}(\Omega)$ and  
$g\in C(\partial \Omega)$ are given. We consider the following 
Dirichlet problem
\begin{equation}\label{equ:Dirichlet}
Lu = f  \mbox{ on } \Omega \quad\mbox{ and } u = g \mbox{ on } \partial \Omega.
\end{equation}

The existence and uniqueness of the solution to \eqref{equ:Dirichlet} follows from
the general theory of existence and uniqueness of the solution to Dirichlet problems
defined on Lipschitz domains in a Riemannian manifold \cite{mitrea-taylor}. 

\begin{lemma}\label{L-property}
Let $n\geq 2$, and let $\Omega$ be a sub-domain on $\bbS^n$ with a Lipschitz boundary. Let $L = \kappa^2 I - \Delta^*$ for some fixed constant $\kappa \ge 0$ and let $s\geq 2 + n/2$. Then $L$ has the following properties:
\begin{itemize}

\item[(i)] There exists a positive constant $c$ such that
\[
 \| L f \|_{H^{s-2}(\Omega)} \le c \|f\|_{H^s(\Omega)}.
\]
\item[(ii)] There exists a positive constant $c$ such that 
\[
    \inprod{Lf}{f}_{L_2(\Omega)} \geq c \|f\|_{L_2(\Omega)}^2
\]
for all $f\in W_2^s(\Omega)\cap C(\overline{\Omega})$ 
with $f=0$ on $\partial \Omega$.
\item[(iii)] There exists a positive constant $c$ such that 
\[
    \|f\|_{C(\overline{\Omega})} \leq c \|f\|_{C(\partial \Omega)} 
\]
for all $f \in W_2^s(\Omega) \cap C(\overline{\Omega})$
which satisfy $L f =0$ on $\Omega$.
\end{itemize}
\end{lemma}
\begin{proof}

(i)

Suppose $s=m$, where $m$ is an integer. 
Using definition \eqref{def:local Sob 2}
and the fact that $\Delta^\ast = -\nabla^* \nabla$, where $\nabla^*$ denote
the surface divergent on the sphere, we have
\begin{align*}
\|Lu\|^2_{W^{m-2}_2(\Omega)} &= 
   \sum_{k=0}^{m-2} \inpro{\nabla^k Lu}{\nabla^k Lu}_{L_2(\Omega)}  \\
    &= \sum_{k=0}^{m-2} \inpro{\nabla^k(\kappa^2 u - \Delta^\ast u)}
                      {\nabla^k(\kappa^2 u - \Delta^\ast u)}_{L_2(\Omega)}  \\
    &= \sum_{k=0}^{m-2} \kappa^4 \inpro{\nabla^k u}{\nabla^k u}_{L_2(\Omega)}
                  - 2 \kappa^2 \inpro{\nabla^{k+1} u}
                  { \nabla^{k+1} u}_{L_2(\Omega)}  \\
    & \qquad \qquad\qquad             + \inpro{\nabla^{k+2} u}{ \nabla^{k+2} u}_{L_2(\Omega)}  \\
    &\le \max\{\kappa^4, 2\kappa^2, 1\}
           \sum_{k=0}^{m} \inpro{\nabla^k u}{\nabla^k u}_{L_2(\Omega)}                            \\ 
     &\le C\|u\|^2_{W^s_2(\Omega)}.       
\end{align*}
The case that $s$ is a real number follows from interpolation between
bounded operators. 

(ii)
With the assumption on $s$, the Sobolev imbedding theorem for functions
defined on Riemannian manifolds \cite[p.34]{hebey} implies that 
$W_2^s(\Omega)\subset C^2(\Omega)$. 

From Green's first surface identity 
\cite[(1.2.49)]{freeden-book}, or more generally, the first Green's formula for compact, connected,
and oriented manifolds in $\R^{n+1}$ \cite[p.84]{agricola-friedrich},
we find for any $f\in W_2^s(\Omega)\cap C(\overline{\Omega})$
with $f=0$ on $\partial \Omega$ that  
\begin{eqnarray*}
    \lefteqn{ \inprod{(\kappa^2-\Delta^\ast)f}{f}_{L_2(\Omega)} = 
    \kappa^2 \|f\|_{L_2(\Omega)}^2 - \inprod{\Delta^\ast f}{f}_{L_2(\Omega)} }\\
    & = & \kappa^2 \|f\|_{L_2(\Omega)}^2 + \|\nabla f\|_{L_2(\Omega)}^2
    - \int_{\partial \Omega} f(\bsx) \frac{\partial f(\bsx)}{\partial\nu}
    \rmd\sigma(\bsx) \\
    & = &\kappa^2 \|f\|_{L_2(\Omega)}^2 + \|\nabla f\|_{L_2(\Omega)}^2,
\end{eqnarray*}
where $\nabla$ is the surface gradient, $\nu$ the (external) unit normal
on the boundary $\partial \Omega$, and $\rmd\sigma$ the curve element 
of the boundary (curve) $\partial \Omega$. 
From the Poincar{\'e} inequality for a bounded domain on a Riemannian manifold
\cite{saloff-coste}, 
\[
    \|\nabla f\|_{L_2(\Omega)} \geq c \|f\|_{L_2(\Omega)}
\]
for all $f\in W_2^s(\Omega)\cap C(\overline{\Omega})$
with $f=0$ on $\partial \Omega$. Thus
\[
    \inprod{(\kappa^2-\Delta^\ast)f}{f}_{L_2(\Omega)} 
    \geq (c+\kappa^2) \|f\|_{L_2(\Omega)}^2,
\]
from which property (ii) is proved.

(iii)

The property (iii) follows 
from the maximum principle for elliptic PDEs on 
manifolds. From \cite[Theorem~9.3]{pucci-serrin}, 
we know that every $g\in C^{1}(\Omega)$ which 
satisfies 
\[ 
\Delta^\ast g - \kappa^2 g \leq 0  \quad\mbox{on}\ 
\Omega \qquad\mbox{and}\qquad g\geq 0 
\quad\mbox{on}\ \Omega  
\]
in distributional sense satisfies the {\em strong maximum 
principle}\/, that is, if $g(\bsy_0)=0$ for some 
$\bsy_0\in \Omega$ then $g\equiv 0$ in $\Omega$.
In particular, this implies if 
$g\in C^{1}(\Omega)\cap C(\overline{\Omega})$ 
that $g$ assumes its zeros on the boundary.  

In our case $f\in W_2^s(\Omega)\cap C(\overline{\Omega})$,
and since $W_2^s(\Omega)\subset C^2(\Omega)$, we consider
(twice differentiable) classical solutions of $\kappa^2 f - \Delta^\ast f =0$.
From the strong maximum principle we may conclude that every
$f\in W_2^{s}(\Omega)\cap C(\overline{\Omega})$
that satisfies $\kappa^2 f - \Delta^\ast f =0$ has the property
\begin{equation}
    \sup_{\bsx\in \overline{\Omega}}|f(\bsx)| = 
    \sup_{\bsx\in\partial \Omega} |f(\bsx)|,  
\label{supremum-1}
\end{equation}
which establishes property (iii) in the Theorem

This can be seen as follows: Consider 
$f\in W_2^{s}(\Omega)\cap C(\overline{\Omega})$
that satisfies $\kappa^2 f - \Delta^\ast f =0$. 
Let $\bsy_1\in\overline{\Omega}$ 
and $\bsy_2\in\overline{\Omega}$ be such that 
\[
    f(\bsy_1) = \min_{\bsy\in \overline{\Omega}} f(\bsy) \leq 
    f(\bsx) \leq  \max_{\bsy\in \overline{\Omega}} f(\bsy) = f(\bsy_2) 
    \qquad\mbox{for all}\ \bsx\in\overline{\Omega}.
\]
Then 
\begin{equation}
    \sup_{\bsx\in \overline{\Omega}}|f(\bsx)|
    = \left\{\begin{array}{l@{\qquad\mbox{if}\ }l}
    f(\bsy_2) & f\geq 0 \ \mbox{on}\ \overline{\Omega},\\
    -f(\bsy_1) & f\leq 0 \ \mbox{on}\ \overline{\Omega},\\
    \max\{-f(\bsy_1),f(\bsy_2)\} &
    \mbox{$f$ assumes negative and positive values}.
    \end{array}\right. 
\label{supremum-2}
\end{equation}
If $f(\bsy_1)\leq 0$, consider $g_1(\bsx):= f(\bsx)-f(\bsy_1)$. 
Then $g_1(\bsy_1)=0$ and $g_1(\bsx)\geq 0$ on $\overline{\Omega}$, 
and we have
\[
    (\Delta^\ast - \kappa^2) g_1 
    = (\Delta^\ast - \kappa^2 ) f + \kappa^2 f(\bsy_1)
    = \kappa^2 f(\bsy_1) \leq 0.
\]
Thus the strong maximum principle implies that $g_1$
assumes its zeros on the boundary and hence 
$\bsy_1\in\partial \Omega$. If $f(\bsy_2)\geq 0$, 
consider $g_2(\bsx):= f(\bsy_2)-f(\bsx)$. Then $g_2(\bsy_2)=0$ and 
$g_2(\bsx)\geq 0$ on $\overline{\Omega}$, and we find
\[
    (\Delta^\ast - \kappa^2) g_2 
    = - \kappa^2 f(\bsy_2) - (\Delta^\ast - \kappa^2 ) f 
    = - \kappa^2 f(\bsy_2) \leq 0.
\]
Thus the strong maximum principle implies that $g_2$
assumes its zeros on the boundary and hence 
$\bsy_2\in\partial \Omega$. Thus (\ref{supremum-2})
implies (\ref{supremum-1}).
\end{proof}

We now discuss a method to construct an approximate solution to
the Dirichlet problem~\ref{equ:Dirichlet} using radial basis functions. 
Assume that the values of the functions $f$ and $g$ are given
on the discrete sets $X_1:=\{\bsx_1,\bsx_2,\ldots,\bsx_M\} \subset \Omega$ 
and $X_2:=\{\bsx_{M+1},\ldots,\bsx_N\} \subset \partial \Omega$, 
respectively. Furthermore, assume that the local mesh norm 
$h_{X_1,\Omega}$ of $X_1$ and the mesh norm $h_{X_2,\partial \Omega}$
of $X_2$ along the boundary $\partial \Omega$ 
(see (\ref{boundary-mesh-norm}) below) are sufficiently small.
We wish to find an approximation of the solution 
$u\in W_2^s(\Omega)\cap C(\overline{\Omega})$ of 
the {\em Dirichlet boundary value problem}\/
\[
    L u = f \quad\mbox{on}\ \Omega\qquad\mbox{and}\qquad 
      u = g \quad\mbox{on}\ \partial \Omega.
\]

Let $\Xi=\Xi_1\cup\Xi_2$ with
$\Xi_1 := \{\delta_{\bsx_j}\circ L \,:\, j=1,2,\ldots,M\}$
and $\Xi_2 := \{ \delta_{\bsx_j} \,:\, j=M+1,\ldots,N\}$.

We choose a RBF $\phi$ such that $\calN_\phi = H^s(\bbS^n)$ for some 
$s>2+\lfloor n/2 + 1 \rfloor$. Under the assumption that
$\Xi$ is a set of linearly independent functionals,
we compute the 
RBF approximant $\Lambda_\Xi u$, defined by
\begin{equation}
\label{local-RBF-approx}
    \Lambda_\Xi u = \sum_{j=1}^M \alpha_j L_2\phi(\cdot,\bsx_j)
                   +\sum_{j=M+1}^N \alpha_j \phi(\cdot,\bsx_j),
\end{equation}
in which the coefficients $\alpha_j$, for $j=1,\ldots,N$, are computed from 
the {\em collocation conditions}\/
\begin{eqnarray}
    L(\Lambda_\Xi u)(\bsx_j) & = & f(\bsx_j), \qquad j=1,2,\ldots,M,
    \label{collocation-local-1}\\
    \Lambda_\Xi u (\bsx_j) & = & g(\bsx_j), \qquad j=M+1,\ldots,N.
    \label{collocation-local-2}
\end{eqnarray}

We want to derive $L_2(\Omega)$-error estimates between the approximation
and the exact solution, which is stated in the following theorem. 
\begin{theorem}\label{bvp-theorem}
Let $L = \kappa^2 I - \Delta^*$ for some fixed constant $\kappa \ge 0$
and let $s\geq 2 + \lfloor n/2+1 \rfloor$. Consider the Dirichlet boundary value problem  
\[
    L u = f \quad\mbox{on}\ \Omega
    \qquad\mbox{and}\qquad    
      u = g \quad\mbox{on}\ \partial \Omega,
\]
where we assume that the unknown solution $u$ is in 
$W_2^{s}(\Omega)\cap C(\overline{\Omega})$ and that 
$f\in W_2^{s-2}(\Omega)$ and $g\in C(\partial \Omega)$.
Assume that $f$ is given on the point set \linebreak
$X_1 = \{\bsx_1,\bsx_2,\ldots,\bsx_M\}\subset \Omega$ with 
sufficiently small local mesh norm $h_{X_1,\Omega}$, and
suppose that $g$ is given on the point set  
$X_2 = \{\bsx_{M+1},\ldots,\bsx_N\}\subset \partial \Omega$
with sufficiently small mesh norm $h_{X_2,\partial \Omega}$. Let $\phi$ be a 
positive definite zonal continuous kernel of the form (\ref{phi-expansion}) 
for which
\begin{equation}\label{cond:whphi}
       a_\ell \sim (1+\lambda_\ell)^{-s}.
\end{equation}       
Let $\Lambda_{\Xi} u$ denote the RBF approximant (\ref{local-RBF-approx})
which satisfies the collocation conditions (\ref{collocation-local-1})
and (\ref{collocation-local-2}). Then 
\begin{equation}
\label{local-L2-error}
    \| u - \Lambda_{\Xi} u \|_{L_2(\Omega)} 
    \leq c \max\{h_{X_1,\Omega}^{s-2},h_{X_2,\partial \Omega}^{s-n/2}\} 
    \|u\|_{W_2^s(\Omega)}. 
\end{equation} 
\end{theorem}

Our general approach follows the one discussed in 
\cite{franke-schback-1998a}, \cite{franke-schback-1998b}, and 
in \cite[Chapter~16]{wendland-book} for the case of boundary 
problems on subsets of $\R^n$. In contrast to the approach in 
\cite[Chapter~16]{wendland-book}, where the error analysis is 
based on the power function, we also use the results on functions with
scattered zeros (see Theorem~\ref{thm:HNW12}) locally
via the charts. 


\begin{proof}

{\em Step 1.}
First we prove the following inequality using the ideas from 
\cite[Theorem~5.1]{franke-schback-1998a}.
\begin{equation}
\label{L2-estimate}
    \|u-\Lambda_\Xi u\|_{L_2(\Omega)}  
    \leq \|L u - L (\Lambda_\Xi u)\|_{L_2(\Omega)} 
    + c \|u-\Lambda_\Xi u\|_{C(\partial \Omega)}.
\end{equation}

Since the boundary value problem has a unique solution, 
there exists a function 
$w\in W_2^s(\Omega)\cap C(\overline{\Omega})$ 
such that
\begin{equation}
\label{trick-0}
    Lw = Lu \quad\mbox{on}\ \Omega
    \qquad\mbox{and}\qquad 
    w = \Lambda_\Xi u \quad\mbox{on}\ \partial \Omega.
\end{equation}
From the triangle inequality, 
\begin{equation}
\label{trick-1}
    \|u - \Lambda_\Xi u\|_{L_2(\Omega)}
    \leq \|u - w\|_{L_2(\Omega)}
    + \|w - \Lambda_\Xi u\|_{L_2(\Omega)}
\end{equation} 
Since $L(u-w) = 0$ on $\Omega$ (from (\ref{trick-0})), the property
(iii) and (\ref{trick-0}) imply
\begin{equation}
\label{trick-2}
    \|u - w\|_{L_2(\Omega)}
    \leq c\|u - w\|_{C(\overline{\Omega})}
    \leq c \|u - w\|_{C(\partial \Omega)}
    = c \|u - \Lambda_\Xi u\|_{C(\partial \Omega)}.
\end{equation}
Since $w-\Lambda_\Xi u =0$ on $\partial \Omega$ (from (\ref{trick-0})),
the property (ii) and the Cauchy-Schwarz inequality yield that
\begin{eqnarray*}
    \|w-\Lambda_\Xi u\|_{L_2(\Omega)}^2 
    & \leq & \inprod{L(w-\Lambda_\Xi u)}{w-\Lambda_\Xi u}_{L_2(\Omega)} \\
    & \leq & \|L(w-\Lambda_\Xi u)\|_{L_2(\Omega)}
    \|w-\Lambda_\Xi u\|_{L_2(\Omega)},
\end{eqnarray*}
thus implying 
\begin{equation}
\label{trick-3}
    \|w-\Lambda_\Xi u\|_{L_2(\Omega)}
    \leq \|Lw-L(\Lambda_\Xi u)\|_{L_2(\Omega)}
    = \|Lu -L(\Lambda_\Xi u)\|_{L_2(\Omega)},
\end{equation}
where we have used $Lw=Lu$ on $\Omega$ in the last step.
Applying (\ref{trick-2}) and (\ref{trick-3}) in 
(\ref{trick-1}) gives 
\[
    \|u - \Lambda_\Xi u\|_{L_2(\Omega)}
    \leq c \|u - \Lambda_\Xi u\|_{C(\partial \Omega)}
    + \|Lu -L(\Lambda_\Xi u)\|_{L_2(\Omega)}
\]
which proves (\ref{L2-estimate}).

{\em Step 2}. In this step, we will estimate the first term
in the right hand side of \eqref{L2-estimate}.
By using Theorem~\ref{thm:HNW12}, we obtain

\begin{eqnarray}
    \|L u - L (\Lambda_\Xi u)\|_{L_2(\Omega)} 
    &\leq 
    c h_{X_1,\Omega}^{s-2}\|L u - L(\Lambda_\Xi u)\|_{W_2^{s-2}(\Omega)}
    \nonumber \\
  &\leq c h_{X_1,\Omega}^{s-2}\|u - \Lambda_\Xi u\|_{W_2^{s}(\Omega)},
  \label{operator-L2-estimate}
\end{eqnarray}
where we have used
the fact that $\|Lg\|_{W^{s-2}_2(\Omega)} \le C\|g\|_{W^s(\Omega)}$,
see Lemma~\ref{L-property} part i).

Next, our assumptions on the region $\Omega$ allow us to extend
the function $u\in W^s_2(\Omega)$ to a function $Eu \in W^s_2(\bbS^n)$.
Moreover, since $X\subset \Omega$ and $Eu|_{\Omega} = u|_{\Omega}$,
the generalized interpolant $\Lambda_{\Xi} u$ coincides with
the generalized interpolant $\Lambda_{\Xi} (Eu)$ on $\Omega$.
Finally, the Sobolev space norm on $W^s_2(\bbS^n)$ is equivalent
to the norm induced by the kernel $\phi$ and the generalized
interpolant is norm-minimal. This all gives
\begin{eqnarray}
\|u - \Lambda_\Xi u\|_{W_2^{s}(\Omega)} &=&
\| Eu - \Lambda_\Xi E u\|_{W_2^s(\Omega)} 
\le \| Eu - \Lambda_\Xi E u\|_{W_2^s(\bbS^n)} \nonumber\\
&\le& \|Eu\|_{W_2^s(\bbS^n)} \le C \|u\|_{W_2^s(\bbS^n)},\label{equ:est2}
\end{eqnarray}
which establishes the stated interior error estimate.

{\em Step 3}. In this step, we will estimate the second term
in the right hand side of \eqref{L2-estimate}.
For the boundary estimate, by using Theorem~\ref{thm:HNW12}
for $\partial\Omega$, which is manifold of dimension $n-1$, we obtain
\begin{equation}
\label{boundary-term-estimate}
    \|u-\Lambda_\Xi u\|_{C(\partial \Omega)} \leq 
    c h_{X_2,\partial \Omega}^{s-1/2-(n-1)/2}\,
    \|u - \Lambda_{\Xi} u\|_{W_2^{s-1/2}(\partial\Omega)}
\end{equation}
Using the trace theorem (Theorem~\ref{trace}) and \eqref{equ:est2}, we have
\begin{equation}\label{boundary2}
\|u - \Lambda_{\Xi} u\|_{W_2^{s-1/2}(\partial\Omega)}
\le C\|u - \Lambda_{\Xi} u\|_{W_2^{s}(\Omega)} \le C\|u\|_{W_2^{s}(\Omega)}
\end{equation}
The boundary estimate then follows from \eqref{boundary-term-estimate}--
\eqref{boundary2}. 

The desired estimate will follow from results of all three steps.
\end{proof}
\section{Numerical experiments}\label{numerical-section}
In this section, we consider the following boundary value problem 
on the spherical cap of radius $\pi/3$ centered at the north pole:
\begin{align*}
 L u(\bsx) & := -\Delta^* u(\bsx) + u(\bsx) = f(\bsx),
  \quad \bsx \in G(\bsn;\pi/3), \\
   u(\bsx) & = g(\bsx) \hspace{4.27cm}\bsx \in \partial G(\bsn;\pi/3).
\end{align*}

Let $f$ be defined so that the exact solution is
given by the Franke function \cite{Fra79} defined on the unit
sphere $\bbS^2$. To be more precise, let
\[
  \bsx = (x,y,z) =
  (\sin\theta\cos\phi, \sin\theta\sin\phi,\cos\theta)
  \quad\text{ for }\quad \theta \in [0,\pi],\; \phi \in [0,2\pi).
\]
Then we define
\begin{align*}
 u(\bsx)  &= 0.75\exp\left(-\frac{ (9x-2)^2 + (9y-2)^2 }{4}\right) +
            0.75\exp\left(-\frac{(9x+1)^2}{49} - \frac{9y+1}{10}\right) \\
       & +  0.5\exp\left(-\frac{(9x-7)^2 + (9y-3)^2}{4}\right) -
            0.2\exp\left(-(9x-4)^2 - (9y-7)^2\right)
\end{align*}
and compute the function $f$ via the formula
\[
  f(\bsx(\theta,\phi)) =
    -\frac{1}{\sin\theta} \frac{ \partial }{\partial \theta}
      \left(\sin \theta \frac{\partial u}{\partial \theta} \right) -
        \frac{1}{\sin^2\theta} \frac{\partial^2 u} {\partial \phi^2}
          +  u(\bsx(\theta,\phi)).
\]
A plot of the exact solution $u$ is given in Figure~\ref{fig:exact}.
\begin{figure}[h]
\begin{center}
\scalebox{0.3}{\includegraphics{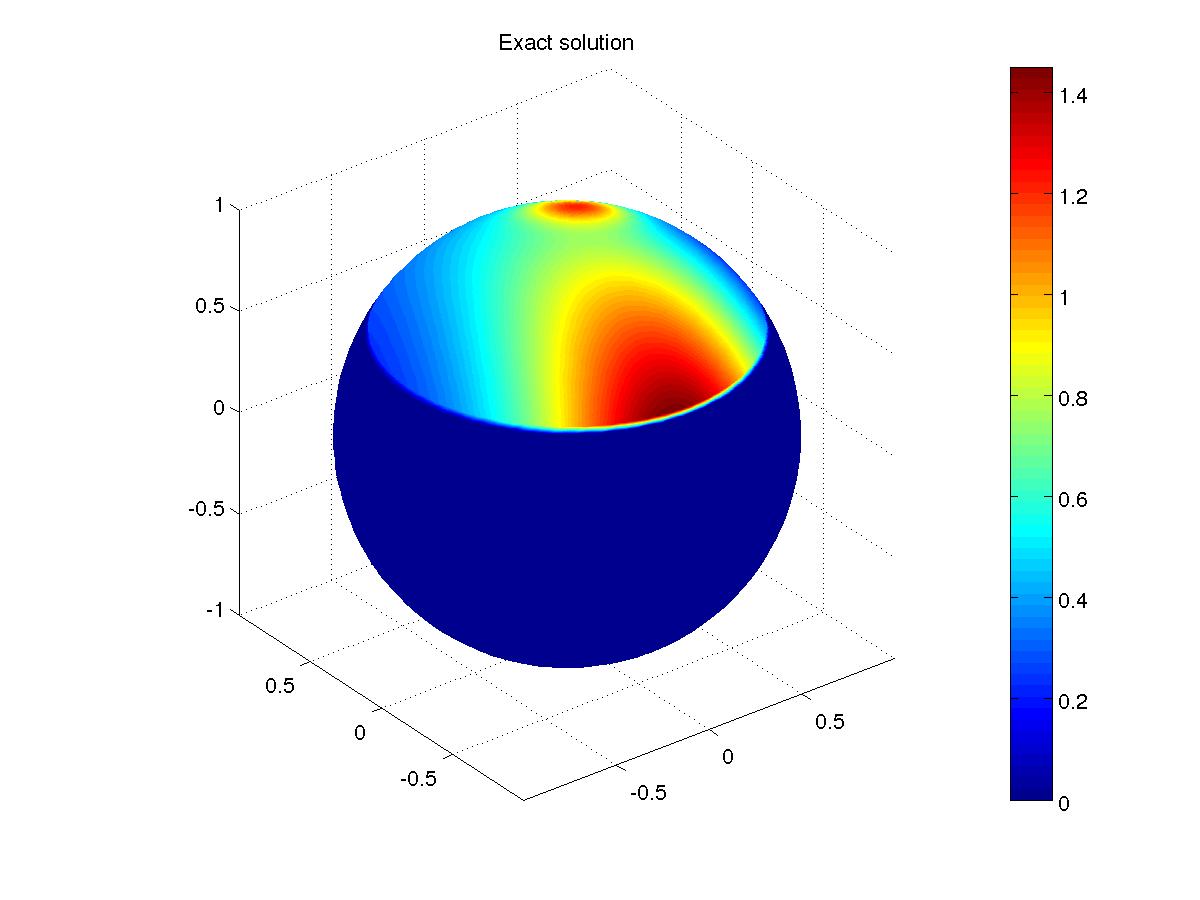}}
\caption{Exact solution}\label{fig:exact}
\end{center}
\end{figure}
Even though the algorithm allows the collocation points to be scattered
freely on the sphere, choosing sets of collocation points distributed roughly uniformly
over the whole sphere significantly improves the quality of the approximate
solutions and condition numbers. To this end, the sets of points used to
construct the approximate solutions are generated using the equal area 
partitioning algorithm \cite{RakSafZho94} adapted to a spherical cap. 

The RBF used is
\[
\psi(r) = (1-r)^8_{+} (1+8r+25r^2+32r^3)
\]
and
\[
\phi(\bsx,\bsy) = \psi(|\bsx - \bsy|) = \psi(\sqrt{2-2\bsx \cdot \bsy}).
\]
It can be shown that $\phi$ is a kernel which
satisfies condition \eqref{cond:whphi} with $s=9/2$
(\cite{narcowich-ward}).

The kernel $\phi$ is a zonal function, i.e. $\phi(\bsx,\bsy) = \Phi(\bsx \cdot \bsy)$
where $\Phi(t)$ is a univariate function. For zonal functions, the 
Laplace-Beltrami operator can be computed via
\[
 \Delta^* \Phi(\bsx \cdot \bsy ) = \calL \Phi(t),
  \quad t = \bsx \cdot \bsy,
\]
where
\[
  \calL  =  \frac{ d} {dt} (1-t^2) \frac{d}{dt}
\]
In our case,
\[
\calL \Phi(t) =  -44( \sqrt{2-2t} - 1)^6 
  \left(6t^2-18t-t+12 + \frac{208t^3  -260 t^2 -92t+144} {\sqrt{2-2t}}  \right).
\]
The normalized interior $L_2$ error $\|e\|$ is approximated by an $\ell_2$ error, 
thus in principle we define (note that the area of the cap $G(\bsn;\pi/3)$ is $\pi$)
\begin{align*}
\|e\| & := \left(
 \frac{1}{\pi}
    \int_{G(\bsn;\pi/3)} |u(\bsx) - \Lambda_{\Xi} u(\bsx)|^2 d\bsx  \right)^{1/2} \\
      & = \left(
        \frac{1}{\pi}
          \int_0^{\pi/3} \int_0^{2\pi} |u(\theta,\phi)-\Lambda_{\Xi}u(\theta,\phi)|^2
              \sin\theta d\phi d\theta \right)^{1/2},
\end{align*}
and in practice approximate this by the midpoint rule, 
\[
\left(\frac{1}{\pi} \frac{2\pi^2}{3|\calG|}
\sum_{\bsx(\theta,\phi) \in \calG}
|u(\theta,\phi) - \Lambda_{\Xi}u(\theta,\phi)|^2\sin\theta \right)^{1/2},
\]
where $\calG$ is a longitude-latitude grid in the interior of $G(\bsn;\pi/3)$ containing 
the centers of rectangles of size $0.9$ degree times $1.8$ degree and $|\calG| = 67 \times 200 = 13400$.

The supremum error $L^\infty(\partial G(\bsn;\pi/3))$ is approximated by 
$$\|e\|_\infty = \max_{\bsx \in \calG'} | u (\bsx) - \Lambda_{\Xi} u(\bsx)|$$ in which $\calG'$ is a set
of $3000$ equally spaced points on $\partial G(\bsn;\pi/3)$.

As can be seen from in Tables~\ref{tab:int} and \ref{tab:bdy}, the
numerical results show a better convergence rate predicted by Theorem~\ref{bvp-theorem}.
\begin{table}
\begin{center}
\begin{tabular}{|c|c|c|c|}
\hline
$M$  &  $h_{X_1}$ &  $\|e\|$   & $EOC$ \\
\hline
  500 &  0.0733 & 2.9000E-03 &  \\ 
 1000 & 0.0520 & 5.1602E-04 & 5.03 \\ 
 2000 & 0.0366 & 8.6364E-05 & 5.09 \\ 
 4000 & 0.0258 & 1.4596E-05 & 5.08 \\ 
\hline
\end{tabular}
\caption{Interior errors with a fixed number of boundary points $N-M = 200$}\label{tab:int}
\end{center}
\end{table}
\begin{table}
\begin{center}
\begin{tabular}{|c|c|c|c|}
\hline
$N-M$ & $h_{X_2}$ &  $\|e\|_\infty$ & $EOC$ \\
\hline
100 &  0.0272 & 2.7561E-05 &  \\ 
200 & 0.0136 & 7.0789E-08 & 8.60 \\ 
400 & 0.0068 & 1.0812E-10 & 9.35 \\ 
800 & 0.0034 & 8.4499E-13 & 7.00 \\ 
\hline
\end{tabular}
\caption{Boundary errors with a fixed number of interior points $M=1000$}\label{tab:bdy}
\end{center}
\end{table}
\begin{figure}[h]
\begin{center}
\scalebox{0.3}{\includegraphics{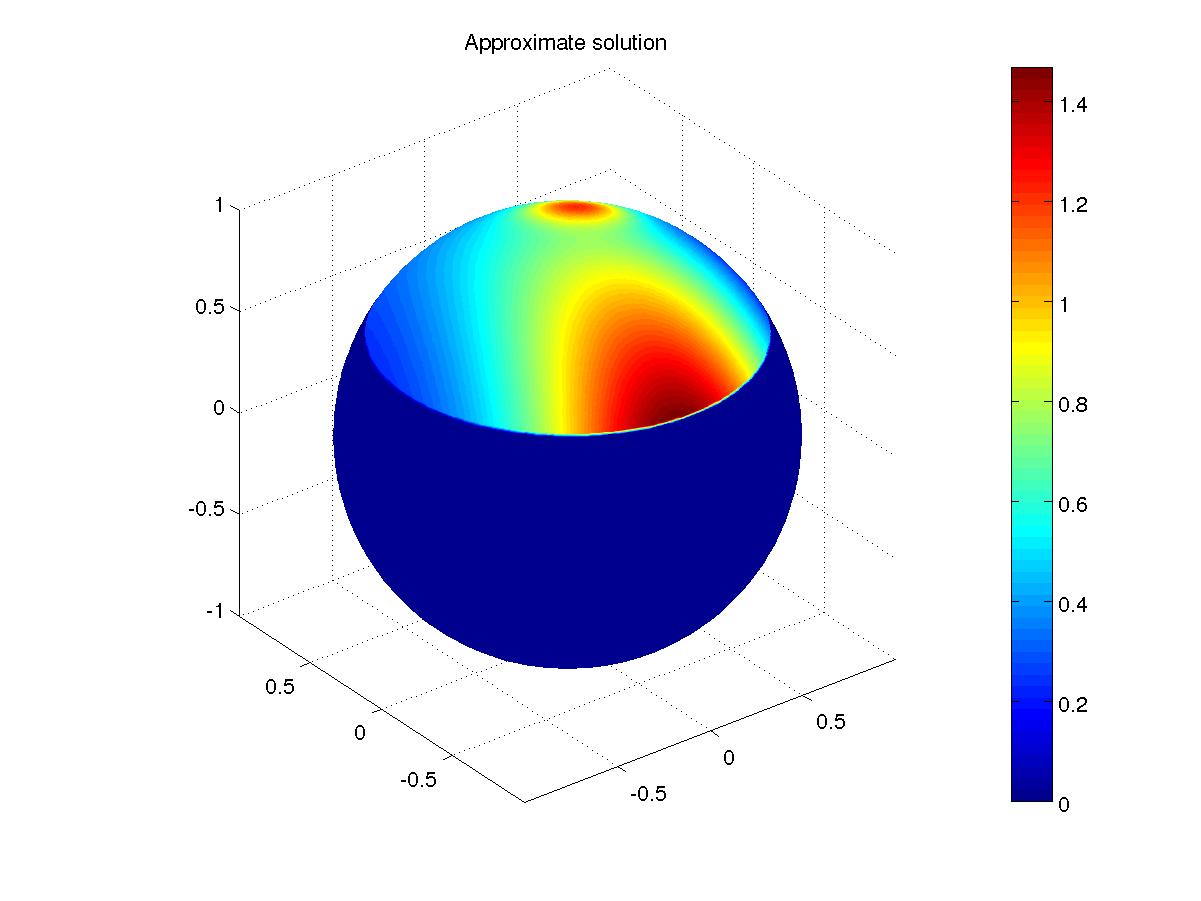}}
\caption{Approximate solution with $M=4000$ and $N=4200$}\label{fig:uX4000_rbf6}
\end{center}
\end{figure}
\begin{figure}[h]
\begin{center}
\scalebox{0.3}{\includegraphics{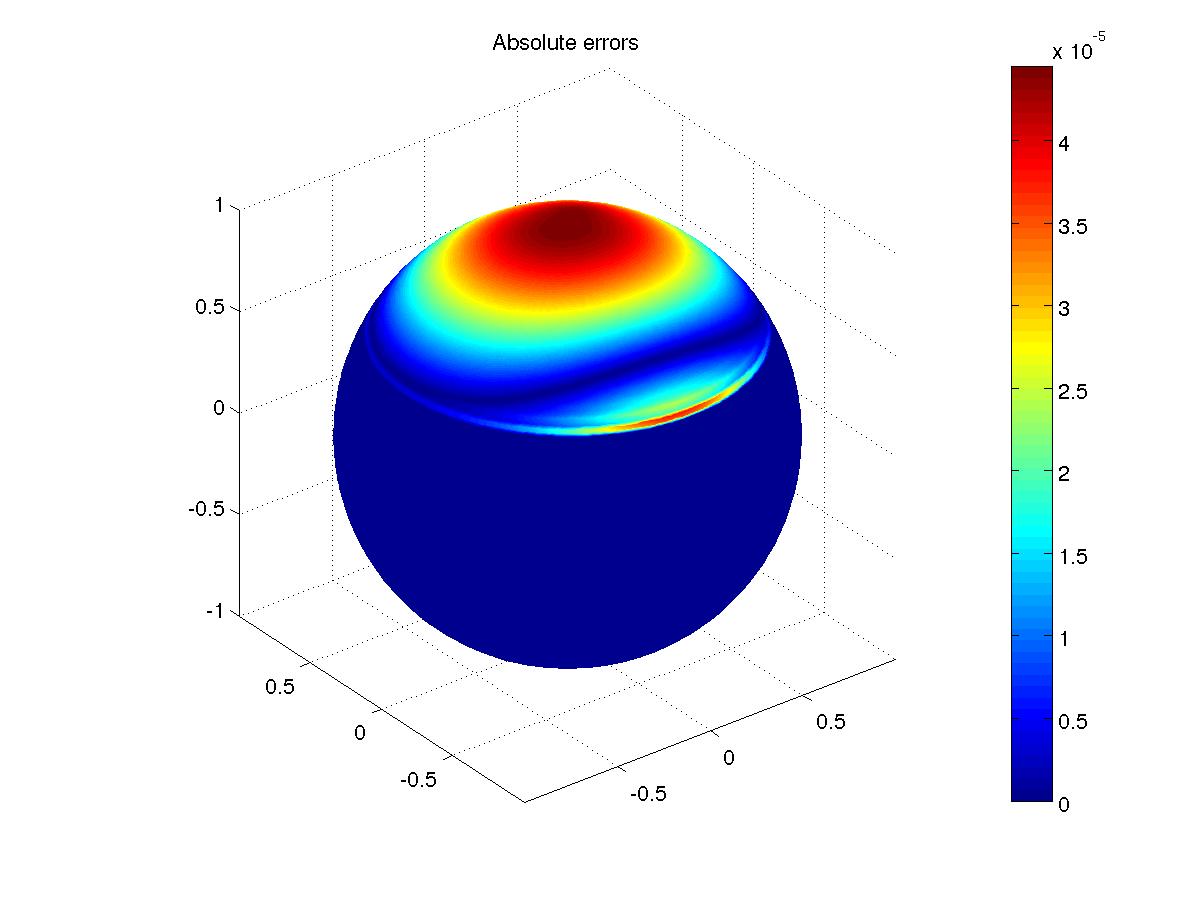}}
\caption{Absolute errors with $M=4000$ and $N=4200$}\label{fig:errX4000_rbf6}
\end{center}
\end{figure}
%
%
\begin{acknowledgement}
The author is grateful to many helpful discussions with Dr. Kerstin Hesse
when writing the earlier version of the paper.
He would also like to thank Professor Francis Narcowich 
for pointing out the recent results on Sobolev bounds for functions 
with scattered zeros on a Riemannian manifold.
\end{acknowledgement}

\end{document}